\documentclass[12pt,a4paper]{amsart}

\newcommand{\cf}{cf.\ }

\usepackage[utf8]{inputenc}
\usepackage[T1]{fontenc}
\usepackage{lmodern}
\usepackage{standard}
\usepackage{amsrefs}
\usepackage{enumerate}
\usepackage[usenames,dvipsnames]{color}
\usepackage[colorlinks=true,linkcolor=Red,citecolor=Green]{hyperref}

\numberwithin{theorem}{section}

\DeclareMathOperator{\F}{\mathcal{F}}
\DeclareMathOperator{\E}{\mathcal{E}}
\DeclareMathOperator{\rk}{rk}
\DeclareMathOperator{\schwartz}{\mathcal{S}}

\newcommand{\symb}{\Gamma}
\newcommand{\symbcl}{\Gamma_{\mathrm{cl}}}
\newcommand{\symbclsg}{\mathrm{SG}_{\cl(\xi)}}
\newcommand{\symbsg}{\mathrm{SG}}
\newcommand{\calc}{\Op \Gamma}
\newcommand{\calccl}{\Op \Gamma_{\mathrm{cl}}}
\newcommand{\calcsg}{\Op \symbsg}

\newcommand{\sobsg}{H_{\mathrm{SG}}}
\newcommand{\WFiso}{\WF_{\mathrm{iso}}}

\newcommand{\hamvf}{\mathsf{H}}

\newcommand{\cl}{\mathrm{cl}}

\newcommand{\Xray}{\mathsf{X}}

\newcommand{\phired}{\phi_{\mathrm{red}}}

\newcommand{\rd}{F}
\newcommand{\pu}{\widetilde{U}_0}

\title[Recurrence of singularities]{Recurrence of singularities for second order isotropic pseudodifferential operators}

\author[M. Doll]{Moritz Doll}
\address{Institut für Analysis, Leibniz Universität Hannover, Welfengarten 1, \newline\indent D-30167 Hannover, Germany}
\email{doll[AT]math.uni-hannover.de}

\thanks{The author would like to thank Oran Gannot, Elmar Schrohe, and Jared Wunsch for useful discussions. This research was supported by the DFG GRK-1463.}

\begin{document}

\begin{abstract}
    Let $H$ be a self-adjoint isotropic elliptic pseudodifferential operator of order $2$.
    Denote by $u(t)$ the solution of the Schrödinger equation $(i\pa_t - H)u = 0$ with initial data $u(0) = u_0$.
    If $u_0$ is compactly supported the solution $u(t)$ is smooth for small $t > 0$, but not for all $t$.
    We determine the wavefront set of $u(t)$ in terms of the wavefront set of $u_0$ and the principal and subprincipal symbol of $H$.
\end{abstract}

\maketitle

\section{Introduction}
It is well-known that the harmonic oscillator $H_0 = 1/2(\Delta + |x|^2)$ on $\RR^d$ has the property that
for compactly supported initial data $u_0 \in \E'(\RR^d)$, the solution $u(t) = e^{-itH_0}u_0$ to the dynamical Schrödinger equation
is smooth for $t \not \in \pi \ZZ$ and $u(\pi k) = (-iR)^ku$, where $Ru(x) = u(-x)$ is the reflection operator.
In particular, we can calculate the wavefront set of $u(t)$:
\begin{equation}
    \WF(u(t)) =
    \begin{cases}
        (-1)^k\WF(u_0) & t = \pi k, k\in \ZZ,\\
        \emptyset & t \not \in \pi \ZZ.
    \end{cases}
\end{equation}

The isotropic (Shubin) symbols of order $m \in \RR$, $\symb^m(\RR^d)$, are given by functions $a \in \CI(\RR^{2d})$ such that for all $\alpha,\beta \in \NN^d$,
\begin{align*}
    \abs{\pa_x^\alpha \pa_\xi^\beta a(x,\xi)} \lesssim_{\alpha,\beta} \ang{(x,\xi)}^{m-|\alpha|-|\beta|}.
\end{align*}
If $a$ admits an asymptotic expansion in homogeneous terms $a_{m-j}$ of order $m-j$ in $(x,\xi)$,
\begin{align*}
    a \sim \sum_{j>0} a_{m-j},
\end{align*}
then we say $a$ is classical, $a \in \symbcl^m$.
We call a symbol \emph{elliptic} if $|a_m(x,\xi)| > 0$ for $(x,\xi) \not = 0$.

To each function $a$ we associate a pseudodifferential operator by the Weyl quantization,
\begin{align*}
    (a^w(x,D)u)(x) = (2\pi)^{-d} \int_{\RR^{2d}} e^{i(x-y)\xi} a( (x+y)/2,\xi) u(y)\,dy\,d\xi,\quad u\in \schwartz.
\end{align*}
We set $\calc^m = \{a^w(x,D)\colon a \in \symb^m\}$, the class of isotropic pseudodifferential operators of order $m$.

Let $p \in \symbcl^2$ be a real-valued classical elliptic isotropic symbol of order $2$ and set
\begin{align*}
    H &= p^w(x,D) \qquad \text{and}\qquad
    H_0 = p_2^w(x,D),
\end{align*}
the pseudodifferential operator and the ``free'' operator\footnote{The notion ``free'' is borrowed from scattering theory.}, respectively.
We consider the dynamical Schrödinger equation:
\begin{equation}\label{eq:schroe}
    \left\{\begin{aligned}
        (i\pa_t - H)u(t) &= 0\\
        u(0) &= u_0.
    \end{aligned}\right.
\end{equation}
We seek to describe the wavefront set of $u(t)$ in terms of the singularities of $u_0$.

We denote the propagator of the equation by $U(t) = e^{-itH}$, similarly $U_0(t) = e^{-itH_0}$ for the free equation.
We proceed in two steps: First, we calculate the wavefront set for the free propagator and then for the reduced propagator
$F(t) = U_0(-t)U(t)$.

Denote by $\hamvf_0(x,\xi) = \pa_\xi p_2 \pa_x - \pa_x p_2 \pa_\xi$ the Hamiltonian vector field associated to the Hamiltonian function $p_2$ and $t \mapsto \exp(t\hamvf_0)$ its flow.
Let $t > 0$ be arbitrary.
We write
\begin{align*}
    \exp(t\hamvf_0)(y,\eta) = (x(t,y,\eta),\xi(t,y,\eta)).
\end{align*}
Let $\Gamma_t = \left\{\eta \in \RR^d\setminus 0 \colon \exp(t\hamvf_0)(0,\eta) \in 0 \times \RR^d\right\}$ and
define the function $\Xi_t : \Gamma_t \to \RR^d$, which is given by $\Xi_t(\eta) = \xi(t,0,\eta)$. It satisfies
\begin{align*}
    \exp(t\hamvf_0)(0,\eta) &= (0,\Xi_t(\eta)).
\end{align*}
Note that $\Xi_t$ is homogeneous of degree one.
Set $G_t = \supp U_0(t)u \times \Xi_t(\Gamma_t)$.

In the following we assume that the manifolds $\Lambda_t = \left\{(x,y,\xi,\eta)\colon \exp(t\hamvf_0)(y,\eta) = (x,\xi)\right\}$
and $0 \times \RR^{2d} \setminus \{0\}$ intersect cleanly for all $t \in \RR$.

\begin{proposition}\label{prop:wf-free}
    Assume that $u \in \schwartz'$ is a tempered distribution.
    The wavefront set of $U_0(t)u$ satisfies
    \begin{align*}
        \WF(U_0(t)u) \cap G_t \subset
        \left\{ \left(x,\Xi_t(\eta)\right) \in G_t \colon y - \pa_\eta \langle x, \Xi_t(\eta)\rangle \bot\, \Gamma_t, (y,\eta) \in \WF(u) \right\}.
    \end{align*}
    If $u \in \E'$, there cannot appear any other singularities:
    \begin{align*}
        \WF(U_0(t)u) \subset \left\{ \left(x,\Xi_t(\eta)\right)\in G_t \colon y - \pa_\eta \langle x, \Xi_t(\eta)\rangle \bot\, \Gamma_t, (y,\eta) \in \WF(u) \right\}.
    \end{align*}
\end{proposition}
\begin{remark}
    If $p_2(x,\xi) = 1/2 (|\xi|^2 + \sum_j \omega_j x_j^2)$ then this proposition follows from Mehler's formula (see Section \ref{sec:ex}).
\end{remark}

We follow the notation of \cite{DGW17} and denote the integral over the flow of $\hamvf_0$ by $\Xray_t$ for any $t \in \RR$:
\begin{align*}
    \Xray_tf = \int_0^t f \circ \exp(s\hamvf_0) ds.
\end{align*}
The wavefront set of the reduced propagator can be completely determined:
\begin{proposition}\label{prop:wf-reduced}
    Let $u \in \schwartz'$ and $t \in \RR$. Then
    \begin{align*}
        \WF(F(t)u) = \left\{ (x + \pa_\xi \Xray_tp_1(0,\xi), \xi) \colon (x,\xi) \in \WF(u) \right\}.
    \end{align*}
\end{proposition}

Combining Proposition~\ref{prop:wf-free} and Proposition~\ref{prop:wf-reduced} yields
\begin{theorem}\label{thm:recurrence}
    Let $u \in \E' + \schwartz$ and $t \in \RR$. The classical wavefront set of $U(t)u$ satisfies
    \begin{align*}
        \WF(U(t)u) \subset \left\{\left(x,\Xi_t(\eta)\right)\in G_t \colon \pa_\eta\langle x, \Xi_t(\eta)\rangle-\pa_\eta\Xray_t p_1(0,\eta)-y \bot\, \Gamma_t, (y,\eta) \in \WF(u) \right\}.
    \end{align*}
\end{theorem}
\subsection*{History}
The usual setting is the Laplacian on $\RR^d$ plus a potential perturbations, that is
\begin{align*}
    H = H_0 + V,
\end{align*}
with $H_0$ the harmonic oscillator and $V = V(x)$ a potential perturbation.
Zelditch~\cite{Zelditch83} proved that for $V \in S^0_{\cl}$ that the singular support of $e^{-itH}u$ is equal to the singular support of $e^{-itH_0}u$ by calculating the Schwartz-kernel.
This was improved by Weinstein~\cite{Weinstein85} to show that the wavefront sets are equal.
Further results for zeroth order perturbations were obtained by Kapitanski, Rodnianski and Yajima~\cite{KapRodYaj97}, Wunsch~\cite{Wunsch99}, Mao and Nakamura~\cite{MaoNak09}.
In the case $V \in S^1_{\cl}$, Doi~\cite{Doi03} and Mao~\cite{Mao14} showed for the harmonic oscillator that singularities are shifted.
For the anharmonic oscillator there appear weaker singularities even for potentials $V \in \CcI$ (\cf Doi~\cite{Doi03}).

Recurrence of singularities for the harmonic oscillator with perturbation by a pseudodifferential operator in the isotropic calculus was proved in \cite{DGW17}.

\subsection*{Outline}
The paper is structed as follows:
In the second section we recall basic properties about the isotropic calculus and the SG-calculus.
The main part of the paper is Section 3. We construct a parametrix as an oscillatory integral for the free propagator for arbitrary large times $t$ and determine the wavefront set
after reducing the oscillatory integral such that the phase is homogeneous of degree one in the fiber-variables.
Section 4 treats the reduced propagator. There, we use a commutator argument in the SG-calculus to determine the classical wavefront set.
The fifth section is a refined version of the stationary phase lemma, where the phase is not linear in the large parameter $\lambda$.
We conclude with two examples, where the leading term is a quadratic form, to illustrate the results.
\section{Global pseudodifferential calculi}
We recall some facts of the SG-calculus and the isotropic calculus.
The SG-calculus is due to Cordes~\cite{Cordes95}, the corresponding wavefront sets at infinity can be found in Coriasco and Maniccia~\cite{CoMa03}.
Shubin (\cf \cite{Shubin78} and References therein) introduced the isotropic calculus. The isotropic wavefront set was added by Hörmander~\cite{Hormander91}. A self-contained introduction into global pseudodifferential calculi can be found in \cite{NiRo10}.

\subsection{Isotropic calculus}
We have defined the isotropic symbols already in the introduction. Now we will review basic properties.
The differential operators of the form
\begin{align*}
    \sum_{|\alpha|+|\beta| \leq m} a_{\alpha,\beta} x^\alpha D^\beta
\end{align*}
are in $\calc^m$. We have the property that the commutator of isotropic symbols is \emph{two} orders lower then the sum of the orders,
therefore we can view operators of order one in a sense as a perturbation. There is a scale of Sobolev spaces adapted to the isotropic
calculus, but we will not need them. These properties are needed to prove Proposition~\ref{prop:wfiso-reduced}, but this was already done in \cite{DGW17}.

For $u \in \schwartz'(\RR^d)$, we define the \emph{isotropic wavefront set} as the subset of $\RR^{2d} \setminus \{0\}$ such that
\begin{align*}
    \WFiso(u) = \bigcap_{\substack{A \in \calccl^0\\Au \in \schwartz(\RR^d)}} \Sigma_0(A),
\end{align*}
where $\Sigma_0(A)$ is the set of characteristic points,
\[\Sigma_0(A) = \left\{(x,\xi) \in \RR^{2d}\setminus \{0\} \colon a_0(x,\xi) = 0\right\}\]

The isotropic wavefront set is an obstruction for a tempered distribution to be in $\schwartz(\RR^d)$, that is
$\WFiso(u) = \emptyset$ if and only if $u \in \schwartz(\RR^d)$.

Now we will investigate the relationship between smoothness and the isotropic wavefront set.
The following Lemma is a special case of Proposition~2.6 by Hörmander~\cite{Hormander91}.
\begin{lemma}\label{lem:wfiso-compact}
    Let $u \in \E'$. Then the isotropic wavefront set is contained in the horizontal space,
    \begin{align*}
        \WFiso(u) \subset \{0\} \times \RR^d.
    \end{align*}
\end{lemma}
It is well-known that $\WFiso(u) \cap \{0\} \times \RR^d = \emptyset$ implies that $u \in \CI$ (\cf \cites{Hormander91,DGW17}).
We refine this result slightly:\footnote{This result was stated in \cite{Nakamura05} for the homogeneous wavefront set}
\begin{lemma}\label{lem:wfiso-smooth}
    Let $u \in \schwartz'$ and $\Gamma \subset \RR^d\setminus\{0\}$ an open cone. If $\WFiso(u) \cap \{0\} \times \Gamma = \emptyset$ then
    $\WF(u) \cap \RR^d \times \Gamma = \emptyset.$
\end{lemma}
\begin{proof}
    Let $a_0 \in \symbcl^0$ such that $a_0 = 1$ in conic neighborhood of $0 \times \Gamma$ and $\supp a_0 \cap \WFiso(u) = \emptyset$.
    By the properties of the isotropic wavefront set, we obtain
    \begin{align*}
        u &= a_0(x,D) u + (1 - a_0(x,D)) u\\
        &= (1- a_0(x,D))u + \schwartz(\RR^d)\\
        &= (2\pi)^{-d}\int e^{ix\xi} (1 - a_0(x,\xi)) \hat{u}(\xi) d\xi + \schwartz(\RR^d).
    \end{align*}
    Choose $b \in S_{\cl}^0$ with $\supp b \subset K \times \Gamma$, for some compact set $K \subset \RR^d$.
    There is an $R > 0$ such that $\{(x,\xi) \in \supp b\colon |\xi| > R\} \subset \supp a_0$.
    Therefore, the symbol of the composition $b(x,D) (1 - a_0(x,D))$ is compactly supported in $(x,\xi)$.
    This implies that $b(x,D) (1 - a_0(x,D)) : \schwartz' \to \schwartz$ and thus
    \[b(x,D) u = b(x,D) (1 - a_0(x,D))u + b(x,D)a_0(x,D) u \in \schwartz(\RR^d).\]
\end{proof}
Note that $u \in \CI$ does not imply that $\WFiso(u) \cap \{0\} \times \RR^d = \emptyset$.
The function $u : x\mapsto e^{ix^3/3}$ is smooth, but not
rapidly decaying
and one can show using the semiclassical description of isotropic wavefront set that $\WF(u) \subset \{0\} \times \RR^d$.
\subsection{SG-calculus}
The SG-calculus (also called scattering calculus for asymptotically Euclidean manifolds) differs from the isotropic calculus by the fact that
taking derivatives in $x$ does not affect the decay in $\xi$ and vice versa.
The SG-calculus is in a way the more natural way of
defining pseudo-differential operators on $\RR^d$, but it is not suited for second order operators differential such as the harmonic oscillator.
\begin{definition}
    Let $m_\psi,m_e$ be real numbers.
    The class $\symbsg^{m_\psi,m_e}(\RR^d)$ consists of functions $a \in \CI(\RR^{2d})$ such that for all multiindices $\alpha,\beta \in \NN^d$ the is an estimate
    \begin{align*}
        \abs{\pa_x^\alpha \pa_\xi^\beta a(x,\xi)} \lesssim_{\alpha,\beta} \ang{\xi}^{m_\psi - |\beta|} \ang{x}^{m_e - |\alpha|}.
    \end{align*}
\end{definition}
We define the corresponding class of SG-pseudodifferential operators:
\begin{align*}
    \calcsg^{m_\psi,m_e} = \left\{ a^w(x,D) \colon a\in\symbsg^{m_\psi,m_e}\right\}.
\end{align*}
The SG-calculus enjoys the following properities:
\begin{enumerate}[(i)]
\item
    $\calcsg = \bigcup_{m_\psi,m_e} \calcsg^{m_\psi,m_e}$ is a $*$-algebra.
\item
    Differential operators of the form 
    \[
        \sum_{\smallabs{\alpha}\leq m_e,\\\smallabs{\beta}\leq m_\psi}
        a_{\alpha,\beta}x^\alpha D^\beta
    \]
    lie in $\calcsg^{m_\psi,m_e}.$
\item
    There are two principal symbol maps $\sigma^\psi,\sigma^e$,
    \begin{align*}
        \sigma^\psi &: \calcsg^{m_\psi,m_e} \to \symbsg^{m_\psi,m_e}/\symbsg^{m_\psi-1,m_e},\\
        \sigma^e &: \calcsg^{m_\psi,m_e} \to \symbsg^{m_\psi,m_e}/\symbsg^{m_\psi,m_e-1},
    \end{align*}
    such that the following principal symbol sequences are exact:
    \begin{align*}
        0 \to \calcsg^{m_\psi-1,m_e} \to \calcsg^{m_\psi,m_e} \overset{\sigma^\psi}{\to} \symbsg^{m_\psi,m_e}/\symbsg^{m_\psi-1,m_e},\\
        0 \to \calcsg^{m_\psi,m_e-1} \to \calcsg^{m_\psi,m_e} \overset{\sigma^e}{\to} \symbsg^{m_\psi,m_e}/\symbsg^{m_\psi,m_e-1}.
    \end{align*}
    We note that for ellipticity one needs a third principal symbol, $\sigma^{\psi,e}$.
\item
    If $A\in\calcsg^{m_\psi,m_e},$ $B\in\calcsg^{m_\psi',m_e'}$, then
    \[[A,B]\in\calcsg^{m_\psi+m_\psi'-1, m_e+m_e'-1}\]
    and satisfies
    \begin{equation*}
        \sigma^\bullet_{m_\bullet+m_\bullet'-1}([A,B])
        =\frac{1}{i}\{\sigma^\bullet_{m_\bullet}(A), \sigma^\bullet_{m_\bullet'}(B)\},
    \end{equation*}
    for $\bullet \in {\psi,e}$ and
    with the Poisson bracket indicating the (well-defined) equivalence
    class of the Poisson bracket of representatives of the equivalence
    classes of each of the principal symbols.
\item
    Every $A\in\calcsg^{0,0}$ defines a continuous linear map on $L^2(\RR^d).$
\item
    The SG-Sobolev spaces, $\sobsg^{s_\psi,s_e}$
    are defined for $s_\psi,s_e \in \RR$ by
    \[
        f \in \sobsg^{s_\psi,s_e} \Longleftrightarrow \ang{x}^{s_e}\ang{D}^{s_\psi}f \in L^2(\RR^d).
    \]
    For all $m_\psi,m_e,s_\psi,s_e\in \RR$ and all
    $A\in\calcsg^{m_\psi,m_e}$,
    \[
        A: \sobsg^{s_\psi,s_e} \to\sobsg^{s_\psi-m_\psi,s_e - m_e}
    \]
    is continuous.
\item\label{it:limitsob}
    The scale of SG-Sobolev spaces satisfies
    \[\displaystyle\bigcap_{m_\psi,m_e} \sobsg^{m_\psi,m_e}=\mathcal{S}(\RR^d),\quad
    \bigcup_{m_\psi,m_e} \sobsg^{m_\psi,m_e}=\mathcal{S'}(\RR^d).\]
\item
    The classical wavefront set of $u \in \schwartz'$ is given by
    \[
        \WF u = \bigcap_{\substack{A \in \calcsg^{0,-\infty}\\Au \in \schwartz}} \Sigma_\psi(A),
    \]
    where $\Sigma_\psi$ is set of points $(x,\xi) \in \RR^d\times \RR^d\setminus \{0\}$ such that $\sigma^\psi(a)(x,\xi) = 0$.
\end{enumerate}
\section{The free propagator}
We start with reviewing the construction of a parametrix for the free propagator $U_0(t) = e^{-itH_0}$ in the FIO calculus of Helffer--Robert
(\cf \cite{Helffer84}*{Chapter 3}).
Let $T > 0$ such that there exists a short-time parametrix $\pu$ of $U_0$ until time $T$ and $\pu$ has the form
\begin{align*}
    \pu(t) = \int e^{i(\phi_2(t,x,\xi) - y\xi)} a(t,x,\xi) \,d\xi,
\end{align*}
where $a \in \C([0,T],\symbcl^0)$ and $\phi_2 \in \C([0,T],\CI(\RR^{2d}))$,
with the following properties for $t \in [0,T]$ and $|(x,\xi)| > 1$:
\begin{itemize}
    \item $\phi_2$ is homogeneous of degree $2$ in $(x,\xi)$,
    \item $\phi_2$ solves the eikonal equation
        \begin{equation*}
            \left\{\begin{aligned}
                \pa_t \phi_2(t) + p_2(x,\pa_x \phi_2) &= 0,\\
                \phi_2(0) &= x\xi,
            \end{aligned}\right.
        \end{equation*}
    \item $\det \pa_x\pa_\xi \phi_2(t) \not = 0$ for $t \in [0,T]$,
    \item $\exp(-t\hamvf_0)(x,\pa_x\phi_2) = (\pa_\xi\phi_2,\xi)$.
\end{itemize}
The short-time parametrix is constructed by solving the eikonal equation and transport equations for the homogeneous terms of the amplitude $a_j$.
The time $T > 0$ depends on the eikonal equation and the transport equation for $a_0$.
Using Borel summation and Duhamel's formula (\cf \cite{Helffer84}*{Proposition 3.1.1}), we obtain that
\begin{align*}
    \pu(t) - U_0(t) \in \CI([0,T],\mathcal{L}(\schwartz'(\RR^d),\schwartz(\RR^d))).
\end{align*}

Let $t_0 > 0$ and write $t_0 = \sum_{j=1}^N t_j$ such that $t_j \in (0,T)$.
Using the group property of $U_0$, we obtain for $u \in \schwartz'$,
\begin{align*}
    U_0(t_0) &= U_0(t_1)\dotsm U_0(t_N)\\
    &\in \pu(t_1)\dotsm\pu(t_N) + \mathcal{L}(\schwartz'(\RR^d),\schwartz(\RR^d)).
\end{align*}
The parametrix $\pu(t_j)$ has kernel
\begin{align*}
    \pu(t_j,z_{j-1},z_j) = \int e^{i(\phi_2(t,z_{j-1},\xi_j) - z_j\xi)} a(t,z_{j-1},\xi_j) \,d\xi_j.
\end{align*}
We write $x = z_0$ and $y = z_N$,
then a parametrix for $t = t_0$ is given by
\begin{align*}
    \pu(t_0) = \int e^{\phi(x,y,\theta)} a(x,\theta) d\theta,
\end{align*}
where
\begin{align*}
    \theta &= (z_1,\dotsc,z_{N-1},\xi_1,\dotsc,\xi_N),\\
    \phi(z_0,z_N,\theta) &= \sum_{j=1}^N \phi_2(t_j,z_{j-1},\xi_j) - z_j\xi_j,\\
    a(s,z_0,\theta) &= \prod_{j=1}^{N-1} a(t_j,z_{j-1}, \xi_j) \cdot a(t_N,z_{N-1},\theta).
\end{align*}
One advantage of the isotropic calculus is that the new phase function $\phi = \phi(x,y,\theta)$ is homogeneous of degree $2$ in all variables for
$|(x,y,\theta)|$ large enough.

\subsection{Classical flow and Lagrangian submanifolds}
Given an Hamiltonian function $p_2$ we associate the flow $t \mapsto \exp(t\hamvf_0)$ and define the set
\begin{align*}
    \Lambda = \left\{(x,y,\xi,-\eta) \colon \exp(t_0\hamvf_0)(y,\eta) = (x,\xi)\right\}.
\end{align*}
We note that $\Lambda$ is a Lagrangian submanifold of $\RR^{2d}$.

In this section, we always assume that $t = t_0$ and omit $t$ from the notation, for instance
we write $\exp(-t_0\hamvf_0)(x,\xi) = (y(x,\xi),\eta(x,\xi))$.

As mentioned in the introduction, we work under the assumption that $\Lambda$ and $0 \times \RR^{2d}$ intersect cleanly, that means that
$\Lambda_0 = \Lambda \cap 0\times \RR^{2d}$ is a smooth manifold and
\begin{align*}
    T_w\Lambda_0 = T_w\Lambda \cap T_w (0\times \RR^{2d}),\quad w \in \Lambda_0.
\end{align*}
At a point $w_0 = (0,0,\xi_0,\eta_0) \in \Lambda_0$ the tangent space is given by
\begin{align*}
    T_{w_0}\Lambda_0 = \left\{ (0,0,\xi,\eta) \in \RR^{4d} \colon \pa_\xi y(0,\xi_0) \xi = 0, \pa_\xi \eta(0,\xi_0) \xi = \eta\right\}.
\end{align*}
The dimension of $\Lambda_0$ is given by a function $(x,\xi) \mapsto e(\xi,\eta)$, which is a locally constant integer with $e \leq d$.
Further,
\begin{align*}
    e(\xi_0,\eta_0) &= \dim T_{w_0}\Lambda_0\\
    &= d - \rk \pa_\xi y(0,\xi_0)\\
    &= \rk \pa_\xi \eta(0,\xi_0).
\end{align*}

The critical set of $\phi(x,y,\theta)$ is defined by
\begin{align*}
    C_\phi = \{(x,y,\theta) \colon \partial_\theta \phi(x,y,\theta) = 0\},
\end{align*}
with
\begin{align*}
    \partial_\theta \phi = 
    \begin{pmatrix}
        \partial_z \phi_2(t_j, z_j, \xi_{j+1}) - \xi_j\\
        \partial_\xi \phi_2(t_j,z_{j-1},\xi_j) - z_j
    \end{pmatrix}.
\end{align*}
The phase function is non-degenerate and (\cf \cite{Duistermaat})
as a direct consequence of the regular value theorem, we see that
\begin{lemma}
    The set $C_\phi$ is a manifold of dimension $2d$.
\end{lemma}
We have a diffeomorphism
\begin{align*}
    \lambda_{\phi} : C_\phi &\to \Lambda,\\
    (x,y,\theta) &\mapsto (x,y,\pa_x\phi(x,y,\theta),\pa_y\phi(x,y,\theta)).
\end{align*}
Since cleanness is preserved under diffeomorphisms the manifolds $C_\phi$ and $0 \times \RR^N$ intersect cleanly and we
denote the intersection by $C_{\phi,0}$ and it is given by
\begin{align*}
    C_{\phi,0} = \left\{(0,0,\theta) \colon \pa_\theta\phi(0,0,\theta) = 0\right\},\\
    T_{(0,0,\theta_0)}C_{\phi,0} = \left\{ (0,0,\delta\theta) \colon \pa_\theta \pa_\theta \phi(0,0,\theta_0) \delta \theta = 0\right\},
\end{align*}
so we conclude that for $(0,0,\theta_0) \in C_{\phi,0}$ and $(0,0,\xi_0,\eta_0) = \lambda_\phi(0,0,\theta_0)$,
\begin{align*}
    N - \rk \pa_{\theta,\theta} \phi(0,0,\theta_0) = \rk\pa_\xi\eta(0,\xi_0) = e(\xi_0,\eta_0).
\end{align*}

The next proposition is implicit in the work of Helffer--Robert:
\begin{proposition}\label{prop:egorov}
    Let $a \in \symb^m$ and $t \in \RR$ arbitrary. Then
    \begin{align*}
        B = e^{it_0H_0} a(x,D) e^{-it_0H_0}
    \end{align*}
    is an isotropic pseudodifferential operator, $B \in \calc^m$ and its principal symbol is given by
    \begin{align*}
        \sigma^m(B)(y,\eta) = \sigma^m(a)(\exp(t_0\hamvf_0)(y,\eta)).
    \end{align*}
\end{proposition}
\begin{proof}
    By Corollary 2.10.7 from \cite{Helffer84} and the parametrix construction, the operator $B$ is an isotropic pseudodifferential operator with principal symbol
    \begin{align*}
        \sigma^m(B)(y,-\pa_y\phi(x,y,\theta)) = \sigma^m(A)(x,\pa_x\phi(x,y,\theta)),
    \end{align*}
    for $(x,y,\theta) \in C_\phi$. Using the diffeomorphism $\lambda_\phi$ and the Definition of $\Lambda$, we see that this is nothing but
    \begin{align*}
        \sigma^m(B)(y,\eta) = \sigma^m(A)(\exp(t_0\hamvf_0)(y,\eta))
    \end{align*}
    as claimed.
\end{proof}
\begin{proposition}\label{prop:wfiso-free}
    Let $u \in \schwartz'(\RR^d)$. The isotropic wavefront set is given by
    \begin{align*}
        \WFiso(U_0(t_0)u) = \exp(t_0\hamvf_0) \WFiso(u).
    \end{align*}
\end{proposition}
\begin{proof}
    It suffices to prove that $\WFiso(U_0(t_0)u) \supset \exp(t_0\hamvf_0) \WFiso(u)$, equality follows from time-reversal.

    Let $(x_0,\xi_0) \not \in \WFiso(U_0(t_0)u)$. Then there is a $Q \in \calccl^0$ such that $\sigma^0(Q)(x_0,\xi_0) = 1$ and
    $QU_0(t_0)u \in \schwartz$. This implies by Proposition~\ref{prop:egorov}
    \begin{align*}
        Pu &= U_0(-t_0) Q U_0(t_0) u \in \schwartz,
    \end{align*}
    and $\sigma^0(P)(y,\eta) = \sigma^0(Q)(\exp(t_0\hamvf_0)(y,\eta))$. Set $(y_0,\eta_0) = \exp(-t_0\hamvf_0)(x_0,\xi_0)$.
    Then $\sigma^0(P)(y_0,\eta_0) = 1$ and therefore $(y_0,\eta_0) \not \in \WFiso(u)$.
\end{proof}

\subsection{Recurrence of singularities}
Now, we investigate the recurrence of classical singularities, which is more delicate.
We define the reduced phase function $\phired$ by
\begin{align*}
    \phired(t_0,x,y,\eta) = x \Xi_{t_0}(\eta) - y\eta.
\end{align*}
\begin{proposition}\label{prop:parametrix-free}
    The propagator at time $t=t_0$ is locally given by
    \begin{align*}
        U_0(t_0,x,y) = \int_{\Gamma} e^{i\phired(t_0,x,y,\eta)} \tilde{a}(t_0,x,y,\eta) d\eta
    \end{align*}
    modulo a smoothing operator and $\tilde{a} \in S^0$ is a local Kohn-Nirenberg symbol.
\end{proposition}
To prove this proposition we will first split the parametrix of $U_0(t_0)$ into a sum of oscillatory integrals supported near connected components
of the critical set $C_{\phi,0}$. Then we will reduce the number of fiber-variables similarly as in the case of Fourier integral operators on compact manifolds (\cf Hörmander~\cite{Hormander4}). In the last step, we show that the resulting amplitude satisfies the Kohn-Nirenberg estimates.
\begin{proof}
    We write
    \begin{align*}
        \phi(x,y,\theta) = \phi(0,0,\theta) + x\pa_x\phi(0,0,\theta) + y\pa_y\phi(0,0,\theta)  + \sum_{|\alpha|=2}(x,y)^\alpha f_\alpha(x, y, \theta)
    \end{align*}
    for some smooth functions $f_\alpha$ and set $f(x,y,\theta) = \sum_{|\alpha|=2}(x,y)^\alpha f_\alpha(x, y, \theta)$.
    We set $\psi(x,y,\theta) = x\pa_x\phi(0,0,\theta) + y\pa_y\phi(0,0,\theta)$ and $\phi_0(\theta) = \phi(0,0,\theta).$
    
    By choosing a function $\chi \in \CcI(\RR_+)$ such that $\chi(r) = 1$ for $r < R$ for some $R > 0$,
    we may assume that the phase $\phi(0,0,\theta)$ is homogeneous on $\RR^N$.
    In fact, the operator with kernel
    \begin{align*}
        \int_{\RR^N} e^{i\phi(x,y,\theta)} \chi(|\theta|) a(x,\theta) d\theta
    \end{align*}
    is regularizing, so may replace in the following $a(x,\theta)$ by $(1-\chi(|\theta|) )a(x,\theta)$.
    The set $C_{\phi,0}$ is conic and we can choose a conic partition of unity $\{\chi_j\}$ such that
    $C_{\phi,0} \cap \supp \chi_j$ is a connected manifold of dimension $N-e_j$.
    From now on we restrict our considerations to one $\chi_j$.

    After a linear transformation, we may assume $(\theta',\theta'') \in \RR^{e_j}\times \RR^{N-e_j}$ on $\supp \chi_j$ such that
    \begin{align*}
        \rk \pa_{\theta''\theta''}\phi_0(\theta) = N - e_j.
    \end{align*}
    Homogeneity of $\phi_0(\theta)$ implies that $\phi_0(\theta) = 0$ on $C_{\phi,0}$.
    Using the implicit function theorem to the equation $\pa_{\theta''}\phi_0(\theta) = 0$,
    we obtain a smooth map $g : \RR^{e_j} \to \RR^{N-e_j}$ that is homogeneous of degree $1$ outside a compact set
    such that
    \begin{align*}
        (\theta',\theta'') \in C_{\phi,0} \text{ if and only if } \theta'' = g(\theta').
    \end{align*}
    We introduce new coordinates
    \begin{align*}
        (\vartheta',\vartheta'') = (\theta',\theta'' - g(\theta'))
    \end{align*}
    and the phase function
    \begin{align*}
        \varphi_0(\vartheta) = \phi_0(\theta).
    \end{align*}
    Then $(\vartheta',\vartheta'') \in C_{\phi,0}$ if and only if $\vartheta'' = 0$. So, we have that
    \begin{align*}
        C_{\varphi_0} = \{\pa_\vartheta \varphi_0 = 0\} &= C_{\phi,0}.
    \end{align*}
    There exists a quadratic form $Q = Q(\vartheta'')$ with the same signature as $\pa_{\vartheta''\vartheta''}\varphi_0$ such that
    $Q$ and $\varphi_0$ are equivalent in the sense of \cite{Helffer84}*{Definition 2.10.13} by Proposition 2.10.14 in \cite{Helffer84}.
    Since all coordinate transformations are homogeneous of degree $1$, the amplitude and the functions $\psi$ and $f$ are of the same form as before.

    So, we may assume that $\phi(0,0,\theta)$ depends only on $\theta''$ and $C_{\phi,0} = \RR^{e_j}\times 0$.
    The propagator at time $t = t_0$ becomes
    \begin{align*}
        U_0(t_0) &= \int_{\RR^{e_j}} \int_{\RR^{N-e_j}} e^{i\varphi(x,y,\vartheta',\vartheta'')} a(x,\vartheta',\vartheta'') d\vartheta'd\theta''\\
        &= \int_{\RR^{e_j}} e^{i\psi(x,y,\theta',0)} \tilde{a}(x,y,\theta') d\theta',
    \end{align*}
    with
    \begin{align*}
        \tilde{a}(x,y,\theta') = \int_{\RR^{N-e_j}} e^{i\phi_0(\theta'') + i(\psi(x,y,\theta',\theta'') - \psi(x,y,\theta',0))+ if(x,y,\theta)} a(x,\theta',\theta'')d\theta''.
    \end{align*}

    We now show that $\psi(x,y,\theta',0)$ is nothing but $\phired$ in local coordinates. Since in our adapted coordinates $C_{\phi,0} = \RR^{e_j}\times 0$, we see that
    \begin{align*}
        \psi(x,y,\theta',0) = \psi(x,y,\theta)|_{C_{\phi,0}}.
    \end{align*}
    Using the diffeomorphism $\lambda_\phi : C_{\phi} \to \Lambda$, we see that
    \begin{align*}
        \exp(t_0\hamvf_0)(0,\pa_x\phi(0,0,\theta',0)) = (0,-\pa_y\phi(0,0,\theta')).
    \end{align*}
    By the inverse function theorem, there exists a map $\eta \mapsto \theta'$ such that
    \begin{align*}
        -\pa_y\phi(0,0,\theta'(\eta),0) = \eta.
    \end{align*}
    Thus, $\psi(x,y,\theta'(\eta),0) = \phired(t_0,x,y,\eta).$

    The map $\eta \mapsto \theta'$ is homogeneous of degree $1$.
    Therefore, it only remains to show that $\tilde{a}$ is a Kohn-Nirenberg symbol.
    We define the amplitude
    \begin{align*}
        c(x,y,\theta) = e^{if(x,y,\theta)} a(x,y,\theta)
    \end{align*}
    and we write the $\psi$-phase as
    \begin{align*}
        \psi(x,y,\theta',\theta'') - \psi(x,y,\theta',0) = \langle \theta'', g(x,y,\theta',\theta'')\rangle,
    \end{align*}
    with $g(x,y,\theta',\theta'') = \int_0^1 \pa_{\theta''}\psi(x,y,\theta',\theta'') dt.$
    Note that the functions $c$ and $g$ satisfy
    \begin{align*}
        \abs{\pa_x^\alpha\pa_y^\beta\pa_\theta^\gamma c(x,y,\theta)} &\lesssim_{\alpha,\beta,\gamma} \ang{\theta}^{-|\gamma|} \ang{(x,y)}^{|\alpha|+|\beta| + 2|\gamma|},\\
        \abs{\pa_x^\alpha\pa_y^\beta\pa_\theta^\gamma g(x,y,\theta)} &\lesssim_{\alpha,\beta,\gamma} \ang{\theta}^{-|\gamma|} \ang{x}^{1-|\alpha|}\ang{y}^{1-|\beta|}.
    \end{align*}

    We have to show that $\tilde{a} \in S^0$ on compact sets, in fact we will show that the following estimate holds:
    \begin{align*}
        \abs{\pa_x^\alpha\pa_y^\beta\pa_{\theta'}^\gamma \tilde{a}(x,y,\theta')} \lesssim_{\alpha,\beta,\gamma} \ang{\theta'}^{-|\gamma|} \ang{(x,y)}^{|\alpha|+|\beta| + 2|\gamma|}.
    \end{align*}
    A typical term is of the form
    \begin{align}\label{eq:a-deriv}
        I(x,y,\theta') = \int_{\RR^{N-e}} e^{i\phi_0(\theta'')}e^{i\ang{g,\theta''}} \prod_{j=1}^k \ang{\theta'',\pa^{\kappa_j}_{x,y,\theta'}g} c(x,y,\theta)d\theta''.
    \end{align}
    We use the standard Paley-Littlewood decomposition: Choose a $\tilde{\chi} \in \CcI(\RR^{N-e})$
    such that $\tilde{\chi} \geq 0$ everywhere, $\tilde{\chi}(x) = 1$ for $|x| \leq 1$, and $\tilde{\chi}(x) = 0$ for $x\geq 2$.
    Set $\chi_j(x) = \tilde{\chi}(x/2^j) - \tilde{\chi}(x/2^{j-1})$. Then
    \begin{align*}
        1 &= \tilde{\chi}(x) + \sum_{j=1}^\infty \chi_j(x)\qquad\text{for all } x \in \RR^{N-e}.
    \end{align*}
    For $\lambda = 2^j$, we have
    \begin{align*}
        I(x,y,\theta') &= \sum_{j=1}^\infty \lambda^{N-e} \int_{\RR^{N-e}} e^{i\lambda^2\phi_0(\theta'')}e^{i\lambda\ang{g,\theta''}} \chi_1(\theta'')\prod_{j=1}^k \ang{\lambda\theta'',\pa^{\kappa_j}_{x,y,\theta'}g} c(x,y,\theta',\lambda\theta'')d\theta''\\
        &+ \int_{\RR^{N-e}} e^{i\phi_0(\theta'')}e^{i\ang{g,\theta''}} \tilde{\chi}(\theta'')\prod_{j=1}^k \ang{\theta'',\pa^{\kappa_j}_{x,y,\theta'}g} c(x,y,\theta',\theta'')d\theta''.
    \end{align*}
    
    In order to estimate the sum, we observe that $\theta''$-derivatives of the function
    \begin{align*}
        e^{i\lambda\ang{g,\theta''}} \chi_1(\theta'') \prod_j \ang{\lambda\theta'',\pa^{\gamma_j}_{\theta'}g} c(x,y,\theta',\lambda\theta'')
    \end{align*}
    can be estimated by
    \begin{align*}
        \lambda^{k+|\gamma|}\ang{\theta'}^{-|\gamma|} \ang{x}^{1-|\alpha|}\ang{y}^{1-|\beta|},
    \end{align*}
    where $\kappa = (\alpha,\beta,\gamma) \in \NN^d_x\times\NN^d_y\times\NN^e_{\theta'}$ for the multiindex $\kappa = \sum_{j=1}^k \kappa_j$.
    Using Theorem 7.7.1 from \cite{Hormander1}, we obtain that for all $M > 0$, each summand can be estimated by
    \begin{align*}
        \lambda^{N-e - M} \lambda^{k+|\gamma|} \ang{\theta'}^{-|\gamma|} \ang{x}^{1-|\alpha|}\ang{y}^{1-|\beta|}.
    \end{align*}
    Choosing $M > N-e+k+|\gamma|+1$, we can sum the geometric series, which yields the desired bound.

    For the last term, we have to use the method of stationary phase.
    We note that it suffices to show that
    \begin{align*}
        \abs{\pa_x^\alpha\pa_y^\beta\pa_{\theta'}^\gamma \tilde{a}(x,y,\lambda\theta')} \leq C_{\alpha,\beta,\gamma} 
    \end{align*}
    as $\lambda \to \infty$ for $C_{\alpha,\beta,\gamma}$ independent of $\lambda$.
    We check that $c(x,y,\lambda\theta)$ and all of its derivatives are bounded by some constant independent of $\lambda$. Therefore, Proposition~\ref{prop:stat-phase} proves the claim.
\end{proof}

\begin{proof}[Proof of Proposition~\ref{prop:wf-free}]
    We have constructed a suitable parametrix of $U_0(t)$ in Proposition~\ref{prop:parametrix-free}. By Theorem~8.1.9. in \cite{Hormander3} the wavefront set of the distribution $U_0(t)$ is given by
    \begin{align*}
        \WF(U_0(t)) \subset \left\{(x,y,\Xi_t(\eta),-\eta) \colon \eta \in \Gamma_t,y-\pa_\eta\langle x, \Xi_t(\eta)\rangle \bot\, \Gamma_t\right\}.
    \end{align*}
    By the calculus of wavefront sets, we obtain that for any $(x,\eta) \in \RR^{2d}$ such that $\eta \in \Gamma_t$ and
    $(x,\Xi_t(\eta)) \in \WF(U_0(t)u)$ then $(\pa_\eta \langle x, \Xi(\eta)\rangle,\eta) \in \WF(u)$.

    Now, let $u \in \E'$. Assume that there is a $(x_0,\xi_0) \in \WF(U_0(t)u)$ such that $\xi_0 \not \in \Xi_t(\Gamma_t)$, that is
    there exists $(y_0,\eta_0) \in \RR^{2d}\setminus\{0\}$ such that $\exp(t\hamvf_0)(0,\xi_0) = (y_0,\eta_0)$ and $y_0 \not = 0$.
    By Lemma~\ref{lem:wfiso-smooth}, $(0,\xi_0) \in \WFiso(U_0(t)u)$.
    We have seen that the isotropic wavefront set is shifted by the Hamiltonian flow (Proposition~\ref{prop:wfiso-free}) and therefore $(y_0,\eta_0) = \exp(-t\hamvf_0)(0,\xi_0) \in \WFiso(u)$.
    By Definition of the set $\Gamma_t$, $y_0 \not = 0$, but this contradicts the assumption that $u$ was compactly supported using Lemma~\ref{lem:wfiso-compact}.
\end{proof}
\section{The reduced equation}
The reduced propagator $F(t) = U_0(-t)U(t)$ satisfies
\begin{equation}
    \left\{\begin{aligned}
        \left(\pa_t - U_0(-t)(H - H_0)U_0(t)\right)F(t) &= 0,\\
        F(0) &= \id.
    \end{aligned}\right.
\end{equation}
We define the operator $P(t) = U_0(-t)(H - H_0)U_0(t)$. By Proposition~\ref{prop:egorov}, $P(t) \in \calccl^1$ and the principal symbol is given by
\begin{align*}
    \sigma^1(P(t)) = p_1 \circ \exp(t\hamvf_0).
\end{align*}
\begin{proposition}\label{prop:wfiso-reduced}
    For all $u \in \schwartz'$ and $t \in \RR$,
    \begin{align*}
        \WFiso(\rd(t)u) = \WFiso(u).
    \end{align*}
\end{proposition}
\begin{proof}
    This follows from Lemma 3.1 in \cite{DGW17}.
\end{proof}
\begin{proposition}\label{prop:sghyper}
    Let $a \in \C([0,T],\symbclsg^{1,1})$ be real-valued and assume that there is a bounded set $K \subset \symbsg^{1,1}$ such that for all $|t| \leq T$, $a(t) \in K$
    Consider for $u_0 \in \schwartz'$ the equation
    \begin{equation}\label{eq:sghyper}
        \left\{\begin{aligned}
            (i\pa_t - a(t,x,D))u(t) &= 0,\\
            u(0) &= u_0.
        \end{aligned}\right.
    \end{equation}
    Let $u \in C(\RR, \schwartz')$ be a solution of \eqref{eq:sghyper}. The wavefront set of $u(t)$ is given by
    \begin{align*}
        \WF(u(t)) &= \Psi_t \WF(u_0),
    \end{align*}
    where $\Psi_t$ is the Hamiltonian flow associated to the function $\sigma^\psi(a(t,x,D))$.
\end{proposition}
\begin{remark}
    If we exchange the role $x$ and $\xi$, we can prove that $\WF_e$ evolves according to the Hamiltonian flow $\Psi_t^e$ of the symbol $\sigma^e(a(t,x,D))$,
    \begin{align*}
        \WF_e(u(t)) = \Psi_t^e \WF_e(u_0),
    \end{align*}
    under the assumption that $a$ admits an asymptotic expansion into terms homogeneous in $x$.
\end{remark}
Using the fact that the SG-estimates are weaker then the isotropic estimates (\cf \cite{NiRo10}*{Section 3.1}) we obtain the propagation of singularities result for $F(t)$:
\begin{proof}[Proof of Proposition~\ref{prop:wf-reduced}]
    The symbol  $p(t) \in \symbcl^1 \subset \symbsg^{1,1}$ has principal $\psi$-symbol $(p_1\circ\exp(t\hamvf_0))(0,\xi)$
    and by Proposition~\ref{prop:sghyper}
    the wavefront set is given by \[\WF(F(t)u) = \left\{ (x + \pa_\xi X_tp_1(0,\xi),\xi) \colon (x,\xi) \in \WF(u)\right\}.\]
\end{proof}
Proposition~\ref{prop:sghyper} also follows from \cite{CoMa03}. We give a self-contained proof using a commutator argument (\cf Hörmander~\cite{Hormander3}*{Theorem 23.1.4}).
\begin{proof}
    Let $(x_0,\xi_0) \not \in \WF(u_0)$ then there is a symbol $b \in \symbclsg^{0,-\infty}$ such that
    $b(x_0,\xi_0) = 1$ and $b(x,D)u_0 \in \schwartz$.

    We construct a symbol $q \in \CI([0,T],\symbclsg^{0,-\infty})$ with the following properties
    \begin{itemize}
        \item $[i\pa_t - a(t,x,D), q(t,x,D)] \in \C([0,T],\calcsg^{-\infty,-\infty})$,
        \item $q(0,x,\xi) = b(x,\xi)$,
        \item $\sigma^\psi q(t) = \Psi_t \sigma^\psi b$.
    \end{itemize}
    If we write $q(t,x,\xi) \sim \sum_j q_{-j}(t,x,\xi)$ where $q_{-j}(t) \in \symbclsg^{-j,-\infty}$ is homogeneous of degree $-j$ in $\xi$, 
    we can see that the $\psi$-principal symbol of the commutator is given by
    \begin{align*}
        \sigma^{\psi}([i\pa_t - a(t,x,D), q(t,x,D)]) = i \{\tau + a_0(t,x,\xi), q_0(t,x,\xi)\} = i (\pa_t + H_{a_0})q_0,
    \end{align*}
    where $H_{a_0} = \pa_\xi a_0 \pa_x - \pa_x a_0 \pa_\xi$ is the Hamiltonian vector field of $a_0$. The term of order $-j$ is given by
    \begin{align*}
        i(\pa_t + H_{a_0})q_{-j} + R_j
    \end{align*}
    with $R_j$ depending on $q_0,\dotsc,q_{-j+1}$. By the assumption that $a(t)$ is contained in a fixed bounded set for all $|t| \leq T$,
    the equations
    \begin{equation*}
        \left\{\begin{aligned}
            (\pa_t y,\pa_t \eta) &= (\pa_\xi a_0(t,y,\eta),-\pa_x a_0(t,y,\eta))\\
            (y(0),\eta(0)) &= (x,\xi)
        \end{aligned}\right.
    \end{equation*}
    have a unique solution for time $|t| \leq T$. The map $\Psi_t(x,\xi) = (y(t,x,\xi),\eta(t,x,\xi))$ is the Hamiltonian flow of the principal symbol
    $\sigma^\psi(a(t,x,D))$ and
    defines a symplectomorphism, which is homogeneous in the second component, $(y(t,x,\lambda\xi),\eta(t,x,\lambda\xi)) = (y(t,x,\xi),\lambda \eta(t,x,\xi))$.
    If we set $q_0 = b_0(\Psi_t^{-1}(x,\xi))$ then $q_0$ solves
    \begin{equation*}
        \left\{\begin{aligned}
            (\pa_t + H_{a_0})q_0(t) &= 0,\\
            q_0(0) &= b_0.
        \end{aligned}\right.
    \end{equation*}
    Similarly, we solve the inhomogeneous equations for $q_{-j}$, $j>0$ by
    \begin{align*}
        q_{-j}(t,x,\xi) = b_{-j}(\Psi_t^{-1}(x,\xi)) + i\int_0^t R_j(\Psi_{t-s}^{-1}(x,\xi)) ds.
    \end{align*}
    If we set $q(t,x,\xi) \sim \sum_{j = 0}^\infty q_{-j}(t,x,\xi)$ we obtain a symbol with the desired properties. This implies that
    if $u(t)$ is a solution to \eqref{eq:sghyper} with initial data $u_0$ then
    \begin{equation*}
        \left\{\begin{aligned}
            (i\pa_t - a(t,x,D)) q(t,x,D)u(t) &\in \C([0,T],\schwartz(\RR^d)),\\
            q(0,x,D)u(0) &\in \schwartz(\RR^d).
        \end{aligned}\right.
    \end{equation*}
    Using an energy estimate (\cf Hörmander~\cite{Hormander3}*{Theorem 23.1.2})
    we conclude that $q(t,x,D)u(t) \in H^{s_\psi,s_e}$ for every $s_\psi,s_e \in \RR$
    and thus $q(t,x,D)u(t) \in \schwartz$.
    This implies by the construction that $\Psi_t(x_0,\xi_0) \not \in \WF(u(t))$.

    The whole argument can be carried out if we replace $t$ by $-t$ and therefore we obtain equality of the wavefront sets.
\end{proof}
\section{Stationary phase with inhomogeneous phase function}
We will derive a formula for calculating stationary phase integrals
\begin{align*}
    I(\lambda,y) = \lambda^d\int_{\RR^d} e^{i\lambda^2\phi(x) + i\lambda(\psi(x,y)-\psi(0,y))} a(\lambda,x,y) dx.
\end{align*}
The function $a$ is smooth and satisfies an estimate $|\pa_x^{\alpha}\pa_y^\beta a(\lambda,x,y)|\leq C_{\alpha,\beta} \lambda^{m}$
and we assume that there is a compact set $K \subset \RR^d\times\RR^n$ such that for every $\lambda \in \RR$, $\supp a(\lambda) \subset K$.
The phase functions $\phi \in \CI(\RR^d)$ and $\psi \in \CI(\RR^d\times \RR^n)$ satisfy
\begin{itemize}
    \item $\phi$ and $\psi$ are real-valued,
    \item $\phi(0) = \pa_x\phi(0) = 0$,
    \item $\pa_{xx}\phi(0)$ is non-singular.
\end{itemize}
\begin{proposition}\label{prop:stat-phase}
    For every $\alpha \in \NN^n$,
    \begin{align*}
        |\pa_y^\alpha I(\lambda,y)| \lesssim_\alpha \lambda^m.
    \end{align*}
\end{proposition}
\begin{proof}
    Define the phase function
    \begin{align*}
        \Phi_\mu(x,y) = \phi(x) + \mu (\psi(x,y) - \psi(0,y)),
    \end{align*}
    for $\mu > 0$ small enough the matrix $\pa_{xx}\Phi_\mu$ is invertible and therefore we may apply the regular value theorem to obtain a map
    $(\mu,y)\mapsto x(\mu,y)$ parametrizing $C_{\Phi_\mu} = \{\pa_x \Phi_\mu = 0\}$.
    Expanding $x(\mu,y)$ into powers of $\mu$ yields
    \begin{align*}
        x(\mu,y) = \mu\tilde{x}(\mu,y), \qquad \tilde{x} \in \CI(\overline{\RR_+} \times \RR^n, \RR^d).
    \end{align*}

    The assumptions on $\phi$ imply that
    \begin{align*}
        \Phi_\mu|_{C_{\Phi_\mu}} &= \phi(\mu \tilde{x}(\mu,y)) + \mu(\psi(\mu\tilde{x}(\mu,y),y) - \psi(0, y))\\
        &= \phi(0) + \mu \pa_x\phi(0) + O(\mu^2)\\
        &= O(\mu^2).
    \end{align*}

    Now, we can estimate $I(\lambda,y)$ and its derivatives:
    The only case where derivatives could cause problems is if it falls on the exponential:
    \begin{align*}
        \lambda^{d+l} \int_{\RR^d} e^{i\lambda^2\phi(x)+i\lambda(\psi(x,y)-\psi(0,y))} \prod_{j=1}^l\left(\pa_{y_{i_j}}(\psi(x,y)-\psi(0,y))\right) a(x,y) dx.
    \end{align*}
    We apply the method of stationary phase and see that each term $\pa_y(\psi(x,y) - \psi(0,y))$ is of order $O(\lambda^{-1})$ since
    \begin{align*}
        \pa_y(\psi(x,y) - \psi(0,y))|_{C_{\Phi_\mu}} = \mu\pa_{x,y}\psi(0,y) + O(\mu^2)
    \end{align*}
    and in our case $\mu = \lambda^{-1}$, further the stationary phase eliminates the prefactor of $\lambda^d$ and obtain
    \begin{align*}
        |\pa_y^{\alpha}I(\lambda,y)| \lesssim \lambda^m
    \end{align*}
    as claimed.
\end{proof}
\section{Examples}\label{sec:ex}
We will consider specific cases of $p \in \symbcl^2$ to illustrate the results.
\subsection{Isotropic harmonic oscillator}
We consider the free Hamiltonian
\begin{align*}
    H_0 = \frac{1}{2}\left(\Lap + |x|^2\right),
\end{align*}
with principal symbol $p_2 = 1/2(|x|^2 + |\xi|^2)$.
It is well-known (\cf Grigis--Sjöstrand \cite{GrSj94}*{Chapter 11}) that the propagator is smoothing for $t \not \in \pi \ZZ$ and
\begin{align*}
    e^{-ik\pi H_0} = (-iR)^k,
\end{align*}
where $Rf(x) = f(-x)$ is the reflection operator.
This implies that for $u \in \schwartz'$ such that $\WFiso(u) \subset 0\times \RR^d$, the wavefront set of $e^{-itH_0}$ is given by
\begin{align*}
    \WF(e^{-itH_0}u) =
    \begin{cases}
        (-1)^k\WF(u) & t = \pi k, k \in \ZZ,\\
        \emptyset & t \not \in \pi\ZZ.
    \end{cases}
\end{align*}
Proposition~\ref{prop:wf-reduced} then implies
\begin{corollary}
    Let $u \in \schwartz'$ such that $\WFiso(u) \subset 0\times \RR^d$ then
    \begin{align*}
        \WF(e^{-i\pi kH}u) = \left\{ (-1)^k \left(x + \pa_\xi (\Xray_{\pi k} p_1)(0,\xi), \xi\right) \colon (x,\xi) \in \WF(u)\right\}
    \end{align*}
    and $\WF(e^{-itH}u) = \emptyset$ if $t \not \in \pi \ZZ$.
\end{corollary}
This was already proved in \cite{DGW17} using an explicit parametrix of the reduced propagator.
\subsection{Anisotropic harmonic oscillator}
Now we take the principal symbol $p_2 \in \symbcl^2(\RR^d)$ with
\begin{align*}
    p_2 = \frac{1}{2}\left(|\xi|^2 + \sum_{j=1}^d \omega_j^2x_j^2\right).
\end{align*}
The Hamiltonian flow of $p_2$ is given by
\begin{align*}
    x_j(t) &= \cos(\omega_j t) x_j(0) + \frac{\sin(\omega_j t)}{\omega_j} \xi_j(0),\\
    \xi_j(t) &= \cos(\omega_j t) \xi_j(0) + \omega_j\sin(\omega_j t) x_j(0).
\end{align*}
Again by Mehler's formula, we have an explicit solution operator:
\begin{align*}
    U_0(t) = \int e^{i(\phi(t,x,\eta) - y\eta)} a(t) \,d\eta
\end{align*}
with
\begin{align*}
    \phi(t,x,\eta) = \sum_{j=1}^d \frac{1}{\cos(\omega_j t)} \left( x_j \eta_j - 1/2\sin(\omega_j t) (\omega_j x_j^2 + \omega^{-1}_j \eta_j^2)\right)
\end{align*}
and $a(t) = \prod_{j=1}^d |\cos(\omega_j t)|^{1/2}$.

For simplicity, assume that $d = 2$ and $\omega_j = j$. Then the flow is periodic with minimal period
$2\pi$ and the propagator is given at the recurrence points by
\begin{align*}
    e^{-i\pi\frac{k}2H_0} = (-iR)^k \otimes (i^{-1/2} \F)^k,\qquad k \in \ZZ.
\end{align*}
That means that $e^{-i\pi\frac{k}{2}H_0}u(x,y) = e^{i\pi k/4} (\F_y)^k u( (-1)^k x, y)$. Note that we take the unitary Fourier transform
\begin{align*}
    \F u(\xi) = (2\pi)^{-d/2} \int e^{-ix\xi} u(x) \,dx.
\end{align*}
From this, we indentify the wavefront set of $e^{-itH_0} u$ for compactly supported initial data $u \in \E'$ as follows:
\begin{align*}
    \WF(e^{-itH_0}u) = \begin{cases}
        \left\{ -(x,y,\xi,0)\colon (x,z,\xi,0) \in \WF(u)\text{ for some }z\in \RR\right\}, & t \in \pi/2 + \pi\ZZ,\\
        \left\{ (x,(-1)^k y,\xi,(-1)^k \eta)\colon (x,y,\xi,\eta) \in \WF(u)\right\}, & t = \pi k, k \in \ZZ,\\
        \emptyset, & t \not \in \frac{\pi}2 \ZZ.
    \end{cases}
\end{align*}
For the first case we have assumed that $(-x,-y) \in \supp e^{-itH_0}u$.

Let $p \in \symbcl^2$ with principal symbol $p_2$ as above and set $H = p^w(x,D)$.
Using Proposition~\ref{prop:wf-reduced} we can calculate the wavefront set of $e^{-itH}u$ in terms of the wavefront set of $u$.
This contrasts the case of potential perturbations, where even smooth compactly supported potential can give rise to new singularities (\cf Zelditch~\cite{Zelditch83}, Doi~\cite{Doi03}) and we can determine the singularities at time $t=\pi/2$, which was not possible in \cite{Doi03}.

\end{document}